\documentclass[a4paper,12pt]{amsart}
\usepackage[dvipdfmx]{graphicx}
\usepackage{amsmath,amssymb,amsfonts,amsthm,color}
\newtheorem{definition}{Definition}[section]
\newtheorem{theorem}[definition]{Theorem}
\newtheorem{lemma}[definition]{Lemma}
\newtheorem{corollary}[definition]{Corollary}

\newtheorem{proposition}[definition]{Proposition}

\newcommand{\C}{\mathbb{C}}
\newcommand{\R}{\mathbb{R}}

\newcommand{\M}{\mathbb{M}}

\newcommand{\N}{\mathbb{N}}

\begin{document}

\title{The $p$-norm of some matrices}

\author{Masaru Nagisa}
\address[Masaru Nagisa]{Department of Mathematics and Informatics, Faculty of Science, Chiba University, 
Chiba, 263-8522,  Japan: \ Department of Mathematical Sciences, Ritsumeikan University, Kusatsu, Shiga, 525-8577,  Japan}
\email{nagisa@math.s.chiba-u.ac.jp}

\maketitle

\begin{abstract}
We compute the operator $p$-norm of some $n\times n$ complex matrices, which can be seen as bounded linear operators on 
the $n$ dimensional Banach space $\ell^p(n)$.
The notion of logarithmic affine matrices is defined, and for such a matrix its $p$-norm is computed exactly.
In particular, a matrix $A=\begin{pmatrix} 8 & 1 & 6 \\ 3 & 5 & 7 \\ 4 & 9 & 2 \end{pmatrix}$ which corresponds to a magic square 
belongs to the class of logarithmic affine matrices, and its $p$-norm is equal to $15$ for any $p\in [1,\infty]$.

\medskip\par\noindent
AMS subject classification: Primary 47A30, Secondary 15B05.

\medskip\par\noindent
Key words: $p$-norm, Riesz-Thorin interpolation, logarithmic affine matrix, circulant matrix, norm inequality.

\end{abstract}

\section{Introduction}
Let $A$ be an $n\times n$ matrix with complex entries.
We can regard $A$ as a bounded linear operators on the $n$ dimensional Banach space $\ell^p(n)$ $(1\le p \le\infty)$.
Our object is to compute  or to estimate the operator norm of $A$ for $p\in [1,\infty]$, that is,
\[    \|A\|_{p,p} = \max \{ \|A\xi\|_p : \xi \in \ell^p(n), \|\xi\|_p=1 \},  \]
where $\|\xi\|_p$ means the vector $p$-norm of $\xi \in \ell^p(n)$ (see Section 2).

In many cases, to compute the norm $\|A\|_{p,p}$ is difficult.
But using Riesz-Thorin theorem (see Proposition 3.1), we can get some estimation for their norms.
Put $f(p)= \|A\|_{p,p}$. 
Then, by Riesz-Thorin theorem, it follows that there exists a $p_0\in [1, \infty]$ such that $f(p)$ is decreasing on $[1,p_0]$
and increasing on $[p_0,\infty]$ (see Proposition 3.2).
Using this $p_0$, we can get the following estimation:
\[   \|A\|_{p,p} \le \begin{cases} \|A\|_{1,1}^{(p_0-p)/p(p_0-1)} \|A\|_{p_0,p_0}^{p_0(p-1)/p(p_0-1)}  & 1\le p < p_0  \\
                          \|A\|_{p_0,p_0}^{p_0/p} \|A\|_{\infty, \infty}^{(p-p_0)/p}  & p_0 \le p \le \infty \end{cases} .  \]
If $A$ is self-adjoint, then $p_0= 2$  and we can get the estimation in Corollary 3.3.                          

We introduce the notion of logarithmic affine matrix. 
We call $A\in \M_n(\C)$  logarithmic affine (in short, $A\in LA(n)$) if it satisfies the following condition:
\[    \|A\|_{p,p} = \|A\|_{1,1}^{1/p} \|A\|_{\infty,\infty}^{1-1/p} \text{ for any } p\in [1,\infty].   \]
We investigate some matrices belongs to $LA(n)$.
It is proved that, for a matrix $A$ satisfying a special condition,  $\|A\|_{p,p}$ is constant for any $p\in[1,\infty]$ (Theorem 3.5).  
As an application, it is proved that, for the following matrices corresponding to magic squares:
\[  \|\begin{pmatrix} 8 & 1 & 6 \\ 3 & 5 & 7 \\ 4 & 9 & 2 \end{pmatrix}\|_{p,p} =15 \text{ and }
     \|\begin{pmatrix} 1 & 2 & 15 & 16 \\ 13 & 14 & 3 & 4 \\ 12 & 7 & 10 & 5 \\ 8 & 11 & 6 & 9 \end{pmatrix}\|_{p,p} =34  \]
for any $p\in[1,\infty]$.
We also estimate the $p$-norm of circulant matrices.
L. Bouthat, A. Khare, J. Mashreghi and F. Morneau-Gu\'erin \cite{Bouthat} and K. R. Sahasranand \cite{Sahasranand} studied
precise estimation of $p$-norm for some special circulant matrices.

In section 4, we consider a block matrices with the special form, and compute its $p$-norm exactly.
We remark that some matrices are not logarithmic affine, but its $p$-norm is determined exactly.

\section{Preliminaries}
Let $\xi = \begin{pmatrix} x_1 \\ \vdots \\ x_n \end{pmatrix} \in \C^n$ and 
$A=\begin{pmatrix} a_{11} & \cdots & a_{1n} \\ \vdots & \ddots & \vdots \\ a_{n1} & \cdots & a_{nn} \end{pmatrix} \in \M_n(\C)$.
We denote
\begin{gather*}
   \|\xi\|_p = \begin{cases} (\sum_{i=1}^n |x_i|^p)^{1/p}  & 1\le p<\infty \\ \max\{|x_i|: i=1,\ldots, n\} & p=\infty \end{cases}  \\
\intertext{and}
   \|A\|_{p,q} =\sup\{\|A\xi\|_q : \|\xi\|_p=1 \}=\sup \{ \frac{\|A\xi\|_q}{\|\xi\|_p}: \xi\neq 0\},  
\end{gather*}
where $1\le p, q\le \infty$.

We denote the Banach space $\C^n$ with the norm $\|\cdot\|_p$ by $\ell^p(n)$ $(1\le p \le \infty)$. 
It is known that the dual space of $\ell^p(n)$  $(1\le p <\infty)$ can be identified with $\ell^q(n)$ by the following way:
\[   \langle \xi, \eta \rangle = \sum_{i=1}^n x_i \bar{y_i} , \]
where $\dfrac{1}{p}+\dfrac{1}{q}=1$,  $\xi=\begin{pmatrix} x_1 \\ \vdots \\ x_n \end{pmatrix}\in \ell^p(n)$ and
$\eta=\begin{pmatrix} y_1 \\ \vdots \\y_n \end{pmatrix} \in \ell^{q}(n)$ and $\bar{y}$ means the conjugate of $y$,
and the dual norm $\|\cdot\|_q$ can be defined as follows:
\[  \|\eta\|_q = \max_{\xi \neq 0} \frac{ {\rm Re} \langle \xi, \eta \rangle}{\| \xi\|_p} 
                    = \max_{\xi \neq 0} \frac{ |\langle \xi, \eta \rangle|}{\| \xi\|_p} . \]

For a linear map $A$ from $\ell^{p_1}(n)$ to $\ell^{p_2}(n)$, we can consider the dual linear map $A^*$ from $\ell^{q_2}(n)$ to $\ell^{q_1}(n)$ as follows:
\[    \langle A \xi, \eta \rangle = \langle \xi, A^*\eta\rangle \quad \xi\in \ell^{p_1}(n),  \eta \in \ell^{q_2}(n),  \]
where $\frac{1}{p_1}+\frac{1}{q_1}=1$ and $\frac{1}{p_2}+\frac{1}{q_2}=1$.
By the calculation
\[   \langle A\xi, \eta \rangle = \sum_{i,j=1}^n a_{ij}x_j\bar{y_i} =\sum_{j=1}^n x_j \overline{(\sum_{i=1}^n \bar{a_{ij}}y_i)} =\langle \xi, \overline{{}^tA}\eta \rangle , \]
we have $A^* =\overline{ {}^tA}= \begin{pmatrix} \bar{a_{11}} & \cdots & \bar{a_{n1}} \\ \vdots & \ddots & \vdots \\ \bar{a_{1n}} & \cdots & \bar{a_{nn}} \end{pmatrix}$ .  
Then we also have
\begin{align*}
   \|A\|_{p_1, p_2} & = \max\{ |\langle A\xi, \eta \rangle| : \|\xi\|_{p_1}=1, \|\eta\|_{q_2}=1  \}  \\
       & = \max\{ |\langle \xi, A^* \eta \rangle | : \|\xi\|_{p_1}=1, \|\eta\|_{q_2}=1  \}  \\
       & = \|A^*\|_{q_2, q_1}.
\end{align*}

%
%
\begin{lemma}\label{lem:inequality}
Let $1\le p_1 \le p_2 \le \infty$.
For any $\xi \in \C^n$, we have
\[  \|\xi\|_{p_2} \le \|\xi\|_{p_1} \le n^{\frac{1}{p_1}-\frac{1}{p_2}}\|\xi\|_{p_2}.  \]
In particular, for any $A\in \M_n(\C)\cong B(\C^n)$,
\[   \|A\|_{p.p} \le n^{|\frac{1}{2}-\frac{1}{p}|}\|A\|_{2,2} .  \]
\end{lemma}
\begin{proof}
Let $\xi = \begin{pmatrix} x_1 \\ \vdots \\ x_n \end{pmatrix} (\neq 0) \in \C^n$. 
Then we have
\begin{align*}
  \frac{\|\xi\|_{p_2}}{\|\xi\|_{p_1}} & = (\sum_{i=1}^n  (\frac{|x_i|}{ \sum_{j=1}^n |x_j|^{p_1} )^{1/p_1} } )^{p_2})^{1/p_2}
     = ( \sum_{i=1}^n ( \frac{|x_i|^{p_1} }{\sum_{j=1}^n |x_j|^{p_1} } )^{p_2/p_1} )^{1/p_2}  \\
     &  \le ( \sum_{i=1}^n \frac{|x_i|^{p_1}}{\sum_{j=1}^n |x_j|^{p_1} } )^{1/p_2} =1,
\end{align*}
and
\begin{align*}
  \|\xi\|_{p_1}^{p_1} & = \sum_{i=1}^n |x_i|^{p_1} \le (\sum_{i=1}^n (|x_i|^{p_1})^{p_2/p_1} )^{p_1/p_2} ( \sum_{i=1}^n (1^{p_2/(p_2-p_1)} ))^{(p_2-p_1)/p_2}  \\
    & \le n^{(p_2-p_1)/p_2} \|\xi\|_{p_2}^{p_1}.
\end{align*}
So we can get the desired inequalities.

We denote the unit matrix of $\M_n(\C)$ by $I_n$, which can be seen as the identity map on $\C^n$.
Using this inequalities,
\[ \|A\|_{p,p} = \|I_n A I_n\|_{p,p} \le \|I_n\|_{2,p} \|A\|_{2,2} \|I_n\|_{p,2} = \begin{cases} n^{\frac{1}{p}-\frac{1}{2}} \|A\|_{2,2} &  (p<2) \\
    n^{\frac{1}{2}-\frac{1}{p}} \|A\|_{2,2} & (p\ge2) \end{cases}. \] 
\end{proof}

For $p\in [1,\infty]$,  $A\in \M_n(\C)$ is called  a $p$-isometry if it holds
\[    \|A\xi \|_p = \|\xi\|_p  \quad \text{for any } \xi \in \C^n.  \]
We call $S=(s_{ij}) \in M_n(\C)$ a unitary permutation if there exist a permutation $\sigma$ on $\{1,2,\ldots, n\}$ and 
$\alpha_i\in \C$ with $|\alpha_i|=1$ ($1\le i \le n$) such that
\[  s_{ij} = \begin{cases} \alpha_i & \text{ if } j=\sigma(i) \\ 0 & \text{ if } j\neq \sigma(i) \end{cases}.  \]

%
%
\begin{lemma}\label{lem:isometry}
\begin{enumerate}
\item[$(1)$] If $S$ is a unitary permutation, then $S$ is $p$-isometry for any $p \in [1,\infty]$.
\item[$(2)$] If $S$ is a $p$-isometry for some $p\in [1,\infty] \setminus \{2\}$, then $S$ is a unitary permutation. 
\end{enumerate}
\end{lemma}
\begin{proof}
(1) Since
\[  \|S\xi\|_p = (\sum_{i=1}^n |\sum_{j=1}^n s_{ij}x_j|^p)^{1/p}=  (\sum_{i=1}^n |\alpha_i x_{\sigma(i)}|^p )^{1/p} = (\sum_{i=1}^n |x_i|^p)^{1/p} = \|\xi\|_p, \]  
$S$ is a $p$-isometry.  

(2) If $S$ is a $p$-isometry, then $S^*$ is a $q$-isometry ($1/p +1/q=1$). 
So we may assume that $1\le p<2<q \le \infty$.

Let $e_j\in \C^n$ be a vector with the $j$-th component is one and the other components are zero.
Since 
$S e_j =\begin{pmatrix} s_{11} & \cdots & s_{1n} \\ \vdots & \ddots & \vdots \\ s_{n1} & \cdots & s_{nn} \end{pmatrix}
          \begin{pmatrix} 0\\ \vdots \\ 1 \\ \vdots \\ 0 \end{pmatrix} = \begin{pmatrix} s_{1j} \\ \vdots \\ s_{nj} \end{pmatrix}$, 
we have $\sum_{i=1}^n|s_{ij}|^p =1$ for any $i=1,2,\ldots, n$.
This implies
\[   |s_{ij}|\le 1 \text{ and } \sum_{i,j=1}^n |s_{ij}|^p = n .\]
Since $S^*$ is a $q$-isometry, we also have
\[   \sum_{i=1}^n|s_{ij}|^{q}=1  \text{ and } \sum_{i,j=1}^n |s_{ij}|^{q} = n .  \]
Then
\[     n= \sum_{i,j=1}^n |s_{ij}|^p \le \sum_{i,j=1}^n |s_{ij}|^{q} = n  \]
implies $|s_{ij}|=0$ or $1$.
So each row vector and column vector of $S$ contain only one non-zero component and its modulus is one.
\end{proof}

\vspace{5mm}

%
%
\begin{lemma}\label{lem:boundaries}
For $A=(a_{ij}) \in \M_n(\C)$, we have
\[  \|A\|_{1,1} = \max_{1\le j \le n} \sum_{i=1}^n |a_{ij}| \text{ and }
    \|A\|_{\infty,\infty} = \max_{1\le i \le n} \sum_{j=1}^n |a_{ij}| .  \]
\end{lemma}
\begin{proof}
Since
\begin{align*}
  \|A\xi\|_1 & = \sum_{i=1}^n | \sum_{j=1}^n a_{ij}x_j| \le \sum_{i,j=1}^n |a_{ij}| |x_j| = \sum_{j=1}^n(\sum_{i=1}^n|a_{ij}|)|x_j| \\
           & \le  \sum_{j=1}^n (\max_{1\le j \le n} \sum_{i=1}^n |a_{ij}| ) |x_i| \le  (\max_{1\le j \le n} \sum_{i=1}^n|a_{ij}|) \|\xi\|_1,  
\end{align*}
we have $\|A\|_{1,1} \le \max_{1\le j \le n} \sum_{i=1}^n |a_{ij}|$.

When $ \max_{1\le j \le n} \sum_{i=1}^n |a_{ij}|= \sum_{i=1}^n |a_{ij_0}|$, we choose $\eta={}^t(y_1,y_2,\ldots,y_n)$ as follows:
\[   y_i = \begin{cases} \bar{a_{ij_0}}/|a_{ij_0}| &  \text{ if } a_{ij_0}\neq 0  \\  0  & \text{ if } a_{ij_0} =0 \end{cases}.  \]
Then we have $\|A\eta\|_1 = (\sum_{i=1}^n |a_{ij_0}|) \|\eta\|_1$ and 
\[    \|A\|_{1,1} = \sum_{i=1}^n |a_{ij_0}| =\max_{1\le j \le n}\sum_{i=1}^n |a_{ij}|.   \]

Since $\ell^\infty(n)$ is the dual space of $\ell^1(n)$, we have
\[  \|A\|_{\infty,\infty} = \|A^*\|_{1,1} = \max_{1\le i \le n}\sum_{j=1}^n |a_{ij}| .  \]
\end{proof}

%
%
\begin{lemma}\label{lem:eigenvalue}
Let $\lambda \in \C$, $\xi\in \C^n\setminus \{0\}$,  $A\in \M_n(\C)$ and a unitary permutation $S$ satisfy $A\xi = \lambda S\xi$.
Then we have  $\|A\|_{p,p} \ge |\lambda|$ for any $p\in[1,\infty]$.
\end{lemma}
\begin{proof}
Since
\[  \|A\xi\|_p = \|\lambda S \xi\|_p =|\lambda| \|S\xi\|_p =|\lambda| \|\xi\|_p,  \]
$\|A\|_{p,p}\ge |\lambda|$.
\end{proof}

\section{Logarithmic affine matrices}

The following statement is well-known as Riesz-Thorin theorem (see \cite{Bergh}):

%
%
\begin{proposition}\label{prop:riesz}
Let $1\le p, q, p_i, q_i \le \infty$ $(i=1,2)$ and $\theta\in [0,1]$ with
\[  \frac{1}{p} = \frac{1-\theta}{p_1}+\frac{\theta}{p_2}  \text{ and }
    \frac{1}{q} = \frac{1-\theta}{q_1}+\frac{\theta}{q_2}.  \]
Then we have
\[   \|A\|_{p,q} \le \|A\|_{p_1,q_1}^{1-\theta} \|A\|_{p_2,q_2}^\theta  .  \]
\end{proposition}

\vspace{5mm}

When $p_1=q_1=1$ and $p_2=q_2=\infty$, we have
\[   \|A\|_{p,p} \le \|A\|_{1,1}^{1/p} \|A\|_{\infty,\infty}^{1-1/p} = \|A\|_{1,1} ( \frac{\|A\|_{\infty,\infty}}{\|A\|_{1,1}})^{1-1/p} 
        = \|A\|_{\infty,\infty}(\frac{\|A\|_{1,1}}{\|A\|_{\infty, \infty}})^{1/p}  \]
by $\frac{1}{p} = \frac{1/p}{1} + \frac{1-1/p}{\infty}$.

%
%
\begin{proposition}\label{prop:key}
For any $A\in \M_n(\C)$, we define a function $f$ on $[1,\infty]$ by $f(p) = \|A\|_{p,p}$  $(p\in [1,\infty])$.
Then it  satisfies the following properties:
\begin{enumerate}
\item[$(1)$]  $f$ is continuous on $[1,\infty]$.
\item[$(2)$]  there exists a $p_0\in [1,\infty]$ such that $f$ is decreasing  on $[1,p_0]$ and increasing on  $[p_0, \infty]$.
Moreover $p<p_0$ $($respectively $p>p_0$ $)$ and $f(p)>f(p_0)$ imply $f(p)> f(q)$ for any $q\in (p,p_0)$  
$($respectively $q\in (p_0, p) $ $)$. 
\item[$(3)$]  For $1\le p<q_0<r\le \infty$, if it holds
\[   f(q_0)^{1/p-1/r}=f(p)^{1/q_0-1/r}f(r)^{1/p-1/q_0},    \]
then
\[   f(q)^{1/p-1/r}=f(p)^{1/q-1/r}f(r)^{1/p-1/q}    \]
for any $q\in [p,r]$.
In particular, $f(p)=f(q_0)=f(r)$ implies $f(p)=f(q)$ for any $q\in[p,r]$. 
\end{enumerate}
\end{proposition}
\begin{proof}
We consider the function 
\[    g(t) = \log f(1/t) = \log \|A\|_{1/t, 1/t}  \qquad \text{for any } t\in[0, 1],   \]
where we use the notation $1/0=0^{-1}=\infty$.
For any $s, t, \theta \in[0,1]$ we put $u=(1-\theta)s + \theta t$, that is,
\[        \frac{1}{u^{-1}} = \frac{1-\theta}{s^{-1}} + \frac{\theta}{t^{-1}}  .  \] 
Using Proposition~\ref{prop:riesz}, we have
\begin{align*}
   g( (1-\theta)s+\theta t) & = g(u) = g(1/u^{-1}) = \log \|A\|_{u^{-1},u^{-1}}   \\
   & \le \log ( \|A\|_{s^{-1},s^{-1}}^{1-\theta} \|A\|_{t^{-1},t^{-1}}^\theta )  \\
   & = (1-\theta)g(s) + \theta g(t).
\end{align*}
This means $g$ is convex on $[0,1]$, so $g$ is continuous on $(0,1)$ and
\[   \lim_{t\to 0+}g(t) \le g(0)  \text{ and } \lim_{t\to 1-}g(t) \le g(1).  \]
By the convexity of $g$, there exists a $t_0\in [0, 1]$ such that g(t) is increasing on $(0,t_0)$ and decreasing on $(t_0,1)$, and
$t<t_0$ with $g(t) >g(t_0)$ implies $ g(s)<g(t)$ for all $s\in (t, t_0)$ and also $t>t_0$ with $g(t)>g(t_0)$ implies $ g(s)<g(t) $ for all $s\in (t_0,t)$.

Since $f(p) = \exp g(p^{-1})$ for $p\in[1,\infty]$, $f$ is continuous on $(1,\infty)$ and
\[  \lim_{p\to\infty} f(p) \le f(\infty)=\|A\|_{\infty,\infty} \text{ and } \lim_{p\to 1+} f(p) \le f(1)=\|A\|_{1,1}.  \]
By Lemma~\ref{lem:boundaries}, we assume that $\|A\|_{1,1} = \sum_{i=1}^n |a_{i1}|$.
We choose $\xi = {}^t(1,0,\ldots,0)$ $\in \C^n$.
Then $\|\xi\|_p=1$ and $\|A\xi\|_p = (\sum_{i=1}^n |a_{i1}|^p )^{1/p}$. 
Since $\|A \xi \|_p \to \|A\|_{1,1}$ ($p\to 1+$), 
\[    \|A\|_{1,1,} \le \lim_{p\to 1+} \|A\|_{p,p} = \lim_{p\to 1+}f(p)\le f(1)= \|A\|_{1,1}.   \]
We also assume $\|A\|_{\infty,\infty}= \sum_{j=1}^n|a_{1j}|$.
We choose $y_j\in \C$ as follows:
\[     a_{1j}y_j = |a_{1j}| n^{-1/p}  \quad \text{ for } j=1,2,\ldots,n . \]
Put $\eta = {}^t(y_1,\ldots, y_n)\in \C^n$.
Then $\|\eta\|_p=1$ and 
\[  \sum_{j=1}^n|a_{1j}| \le \|A\eta\|_p\le n^{1/p}(\sum_{j=1}^n|a_{1j}|) .  \]
Since $\|A \eta \|_p \to \|A\|_{\infty,\infty}$ ($p\to \infty$), 
\[    \|A\|_{\infty, \infty} \le \lim_{p\to \infty} \|A\|_{p,p} = \lim_{t\to \infty}f(p)\le f(\infty)= \|A\|_{\infty,\infty}.   \] 
So $f$ is continuous on $[1,\infty]$ and (1) is proved.

Put $p_0=1/t_0$. Then (2) holds.

The condition $f(q_0)^{1/p-1/r}=f(p)^{1/q_0-1/r}f(r)^{1/p-1/q_0}$ is equivalent to
\[    (p^{-1}-r^{-1}) g(q_0^{-1}) =(q_0^{-1}-r^{-1}) g(p^{-1}) + (p^{-1}-q_0^{-1}) g(r^{-1}) ,  \]
and also equivalent to
\[  \frac{g(p^{-1})-g(q_0^{-1})}{p^{-1}-q_0^{-1}} = \frac{g(q_0^{-1})-g(r^{-1})}{q_0^{-1}-r^{-1}} .  \]
By the convexity of $g$, for any $q\in (p,r)$, it follows
\[  \frac{g(p^{-1})-g(q^{-1})}{p^{-1}-q^{-1}} = \frac{g(q^{-1})-g(r^{-1})}{q^{-1}-r^{-1}} , \]
that is,
\[   f(q)^{1/p-1/r}=f(p)^{1/q-1/r}f(r)^{1/p-1/q}.  \]
So (3) is proved.
\end{proof}

%
%
\begin{corollary}\label{cor:symmetry}
When $A=A^*$,  we have $\|A\|_{2,2}\le \|A\|_{p,p} \le \|A\|_{1,1} = \|A\|_{\infty,\infty}$ for any $p\in[1,\infty]$  
and the function $p \mapsto \|A\|_{p,p}$ is decreasing on $[1,2]$ and increasing on $[2,\infty]$ and
\[   \|A\|_{p,p} \le \begin{cases} \|A\|_{1,1}^{(2-p)/p} \|A\|_{2,2}^{2(p-1)/p} \quad & p\in (1,2) \\
            \|A\|_{2,2}^{2/p} \|A\|_{\infty,\infty}^{(p-2)/p}  & p\in (2,\infty) \end{cases} .  \]
\end{corollary}
\begin{proof}
By Lemma~\ref{lem:boundaries}  and Proposition~\ref{prop:key}(2), we have $\|A\|_{p,p}\le \|A\|_{1,1}=\|A\|_{\infty,\infty}$ 
for any $p\in [1,\infty]$.
We chosse $q$ as follows:
\[  \frac{1}{p}+\frac{1}{q}=1. \]
We remark that $\|A\|_{p,p}=\|A^*\|_{q,q}=\|A\|_{q,q}$ and $p<2$ implies $q>2$.
So we have $p_0=2$, where $s_0$ is in the statement of Proposition~\ref{prop:key}(2).

When $p\in (1,2)$,  $\frac{1}{p}=\frac{1-\theta}{1}+\frac{\theta}{2}$ and $\theta=\frac{2(p-1)}{p}$.
By Proposition~\ref{prop:riesz}, we have
\[     \|A\|_{p,p} \le  \|A\|_{1,1}^{(2-p)/p} \|A\|_{2,2}^{2(p-1)/p} .  \]

When $p\in (2, \infty)$,  $\frac{1}{p}=\frac{1-\tau}{2}+\frac{\tau}{\infty}=\frac{1-\tau}{2}$ and $\tau=\frac{p-2}{p}$.
By Proposition~\ref{prop:riesz}, we have
\[     \|A\|_{p,p} \le  \|A\|_{2,2}^{2/p} \|A\|_{\infty,\infty}^{(p-2)/p} .  \]
\end{proof}

By Proposition~\ref{prop:key}, the function $f(p) = \|A\|_{p,p}$ satisfies the following relation:
\[   f(p) \le f(1)^{1/p}f(\infty)^{1-1/p}  \quad \text{for all } p\in[0,\infty].  \]
We call $f$ lagarithmic affine if it satisfies
\[   f(p) = f(1)^{1/p}f(\infty)^{1-1/p}  \quad \text{for all } p\in[0,\infty]. \]
Then the matrix $A\in M_n(\C)$ is called {\it logarithmic affine}, and we denote $A\in LA(n)$.
By definition, for $A=(a_{ij}) \in LA(n)$
\[  \|A\|_{p,p} = \|A\|_{1,1}^{1/p} \|A\|_{\infty,\infty}^{1-1/p} 
= (\max_{1\le j\le n}\sum_{i=1}^n |a_{ij}|)^{1/p}(\max_{1\le i \le n}\sum_{j=1}^n |a_{ij}|)^{1-1/p}.  \]
Next statement easily follows from Proposition~\ref{prop:key}(3):

%
%
\begin{corollary}\label{cor:logaff} 
Let $A\in M_n(\C)$ and $f(p)=\|A\|_{p,p}$,   
If it holds
\[   f(p_0)=f(1)^{1/{p_0}}f(\infty)^{1-1/{p_0}} \text{ for some }  p_0\in (1,\infty),  \] 
then $A$ is logarithmic affine.
\end{corollary}

We remark that a unitary permutation $S$ satisfies $f(p)=\|S\|_{p,p}=1$ for any $p\in[1,\infty]$.
We want to investigate a matrix $A\in M_n(\C)$ which satisfies the function $f(p)=\|A\|_{p,p}$ is constant on $[1,\infty]$.
Such a matrix clearly belongs to $LA(n)$ because
\[    f(1)=f(\infty)=f(p)  \text{ and } f(p) = f(1)^{1/p}f(\infty)^{1-1/p}.   \]

%
%
\begin{theorem}\label{thm:main}
Let $A=(a_{ij}) \in \M_n(\C)$ with $a_{ij}\ge 0$ for any $i,j=1,2,\ldots, n$ and $\xi ={}^t(1, 1, \cdots, 1)$.
If $A\xi = \alpha \xi$ and $A^*\xi = \alpha  \xi$ for some $\alpha\in \C$, then $\|A\|_{p,p}=\alpha$ for any $p\in[1,\infty]$.
\end{theorem}
\begin{proof}
By the assumption $A\xi = \alpha  \xi$ and $A^*\xi = \alpha  \xi$ imply that 
\[     \sum_{j=1}^n a_{ij} =\alpha \text{ and } \sum_{i=1}^n a_{ij} = \alpha ., \]
and $\alpha$ is non-negative because $a_{i,j}\ge 0$.
By Lemma~\ref{lem:boundaries}  we have $\|A\|_{1,1}=\|A\|_{\infty,\infty}=\alpha$.
By Proposition~\ref{prop:key}(2), it follows $\|A\|_{p,p} \le \alpha$ for all $p\in [1,\infty]$.
By Lemma~\ref{lem:eigenvalue}, $\|A\|_{p,p}\ge \alpha$.
So we have $\|A\|_{p,p}=\alpha$ for any $p\in[1,\infty]$.
\end{proof}

\vspace{5mm}

%
%
\noindent
{\bf Example}(1) The following matrices are corresponding to magic squares:
\[  A =\begin{pmatrix} 8 & 1 & 6 \\ 3 & 5 & 7 \\ 4 & 9 & 2 \end{pmatrix}, \quad
     B =\begin{pmatrix} 1 & 2 & 15 & 16 \\ 13 & 14 & 3 & 4 \\ 12 & 7 & 10 & 5 \\ 8 & 11 & 6 & 9 \end{pmatrix}.  \]
Here we call  $C=(c_{ij}) \in M_n(\C)$ a magic square if it satisfies $c_{ij}\ge 0$ and
there exists $K\ge 0$ such that 
\[   \sum_{k=1}^n c_{ik} = \sum_{k=1}^n c_{kj} = K,   \]
for any $i,j =1,2,\ldots,n$.
That is,  $\xi ={}^t(1, 1, \cdots, 1)$ is an eigenvector of $C$ and $C^*$ for an eigenvalue $K$.
So $\|A\|_{p,p}=15$ and $\|B\|_{p,p}=34$ for all $p\in[1,\infty]$ by Theorem~\ref{thm:main}.
In this case $A$ and $B$ are not normal.

\vspace{5mm}

%
%
\noindent
{\bf Example}(2) Let $S=\begin{pmatrix} 0 & 1 & \cdots & 0 \\  0 & 0 & \ddots & \vdots  \\ \vdots & \ddots & \ddots & 1 \\ 1 & \cdots & 0 & 0 \end{pmatrix}$. 
For $\alpha_i \in \C$   $(0 \le i \le n-1)$, we consider a circulant matrix 
\[  A=\sum_{i=0}^{n-1} \alpha_iS^i =
     \begin{pmatrix} \alpha_0 & \alpha_1 & \cdots & \alpha_{n-1} \\
                          \alpha_{n-1} & \alpha_0 & \ddots & \vdots \\
                          \vdots & \ddots  & \ddots & \alpha_1 \\ 
                          \alpha_1 & \cdots & \alpha_2 & \alpha_0 \end{pmatrix}, \]
and we denote this circulant matrix by $A=C(\alpha_0,\alpha_1, \ldots, \alpha_{n-1})$. 

If $\alpha_i\ge 0$ $(0\le i \le n-1)$, then $A$ is a magic square and
\[   \|A\|_{p,p} = \sum_{i=0}^{n-1} \alpha_i   \text{ for all } p\in[1,\infty].  \]
In particular, $A\in LA(n)$.

%
%
\begin{proposition}
Let $A=\sum_{i=0}^{n-1} \alpha_i S^i$, $\alpha_i\in \C$ $(0\le i \le n-1)$. 
Then $A\in LA(n)$ if and only if there exists $\omega \in \C$ such that  $\omega^n=1$ and
\[   A = \beta C(|\alpha_0|, \omega |\alpha_1|, \omega^2|\alpha_2|, \ldots,\omega^{n-1}|\alpha_{n-1}|)  \]
for some $\beta\in \C$ with $|\beta|=1$. 

In particular, $A\in LA(n)$ implies $\|A\|_{p,p} = \sum_{i=1}^n |\alpha_i|$ for any $p\in [1,\infty]$.
\end{proposition}
\begin{proof}
By Lemma~\ref{lem:boundaries}, $\|A\|_{1,1} = \|A\|_{\infty,\infty}= \sum_{i=0}^{n-1} |\alpha_i|$.
So $A\in LA(n)$ is equivalent to $\|A\|_{2,2} =\sum_{i=1}^n |\alpha_i|$ by Corollary~\ref{cor:logaff}.

We can choose a unitary such that
\[   U^*SU = \begin{pmatrix} 1 & 0 & \cdots & 0 \\ 0 & e^{2\pi\sqrt{-1}/n} & \ddots & \vdots \\ \vdots & \ddots & \ddots & 0 \\ 0 & \cdots & 0 & e^{2\pi(n-1)\sqrt{-1}/n} \end{pmatrix}.  \]
Then
\begin{align*}
     & U^*AU  = \sum_{j=0}^{n-1} \alpha_i U^*S^j U  \\
   = & \begin{pmatrix} \sum_{j=0}^{n-1}\alpha_i & 0 & \cdots & 0 \\ 
                     0 &  \sum_{j=0}^{n-1} e^{2\pi j \sqrt{-1}/n} \alpha_j & \ddots & \vdots \\ 
                     \vdots & \ddots & \ddots & 0 \\ 
                     0 & \cdots & 0 & \sum_{j=0}^{n-1} e^{2\pi(n-1)j\sqrt{-1}/n} \alpha_j \end{pmatrix}, 
\end{align*}
and the spectrum of $A$ is equal to $\{\sum_{i=0}^{n-1} \alpha_i \omega^i : \omega \in \C \text{ with } \omega^n=1 \}$.
So we have
\[   \|A\|_{2,2} = \max \{ |\sum_{i=0}^{n-1} \alpha_i \omega^i| : \omega\in \C \text{ with } \omega^n=1\} \le \sum_{i=0}^{n-1} |\alpha_i|.   \]
Since $\|A\|_{2,2} = \sum_{i=0}^{n-1} |\alpha_i|$, there exists $\omega_0 \in \C$ with $\omega_0^n=1$ and
\[   |\sum_{i=0}^{n-1} \alpha_i \omega_0^i| = \sum_{i=0}^{n-1} |\alpha_i|.  \]
So there exists $\beta \in \C$ with $|\beta|=1$ and
\[   \alpha_i\omega_0^i = \beta |\alpha_i|  \quad \text{ for all } i=0,1,\ldots,n-1.  \]
This means  $A=\beta C(|\alpha_0|, \bar{\omega_0}|\alpha_1|, \ldots, \bar{\omega_o}^{n-1}|\alpha_{n-1}|)$.
\end{proof}

\vspace{5mm}

%
%
\noindent
{\bf Example}(3) We consider a matrix of Hankel type as follows:
\[   B = \begin{pmatrix} b_{00}  & b_{01} & \cdots & b_{0,n-1}  \\
                               b_{10} & b_{11} & \cdots & b_{1,n-1}  \\
                               \vdots & \vdots & \ddots &   \vdots  \\
                               b_{n-1,0} & b_{n-1,1} & \cdots & b_{n-1,n-1}  \end{pmatrix},  \]
where  $b_{i,j} = \alpha_k\in \C$  ($k\equiv  i+j  \mod n$), that is,
\[   B = \begin{pmatrix} \alpha_0 & \alpha_1 & \alpha_2 & \cdots & \alpha_{n-1}  \\
                               \alpha_1 & \alpha_2 & \cdots & \cdots & \alpha_0  \\
                               \alpha_2 & \vdots & \ddots &\cdots & \vdots \\
                               \vdots  &  \vdots & \vdots & \ddots &  \alpha_{n-3}  \\
                               \alpha_{n-1} & \alpha_0 & \cdots & \alpha_{n-3} & \alpha_{n-2} \end{pmatrix} . \]
We denote $B=H(\alpha_0, \alpha_1, \ldots, \alpha_{n-1})$.                             
                               
This matrix $B$ can be represented by a product of a permutation matrix and a circulant matrix.
For examples,
\begin{gather*}
     \begin{pmatrix} \alpha_0 & \alpha_1 & \alpha_2 \\ \alpha_1 & \alpha_2 & \alpha_0 \\ \alpha_2 & \alpha_0 & \alpha_1 \end{pmatrix}
   = \begin{pmatrix} 1 & 0 & 0 \\ 0 & 0 & 1 \\ 0 & 1 & 0 \end{pmatrix} 
      \begin{pmatrix} \alpha_0 & \alpha_1 & \alpha_2 \\ \alpha_2 & \alpha_0 & \alpha_1 \\ \alpha_1 & \alpha_2 & \alpha_0 \end{pmatrix},  \\
     \begin{pmatrix} \alpha_0 & \alpha_1 & \alpha_2 & \alpha_3 \\ \alpha_1 & \alpha_2 & \alpha_3 & \alpha_0  \\ 
                         \alpha_2 & \alpha_3 & \alpha_0 & \alpha_1 \\ \alpha_3 & \alpha_0 & \alpha_1 & \alpha_2 \end{pmatrix}
   = \begin{pmatrix} 1 & 0 & 0 & 0 \\ 0 & 0 & 0 & 1 \\ 0 & 0 & 1 & 0  \\ 0 & 1 & 0 & 0 \end{pmatrix} 
      \begin{pmatrix} \alpha_0 & \alpha_1 & \alpha_2 & \alpha_3 \\ \alpha_3 & \alpha_0 & \alpha_1 & \alpha_2 \\ \alpha_2 & \alpha_3 & \alpha_0 & \alpha_1 \\
                          \alpha_1 & \alpha_2 & \alpha_3 & \alpha_0  \end{pmatrix} 
\end{gather*}
and, in general,
\[   H(\alpha_0, \alpha_1, \ldots, \alpha_{n-1}) = H(1,0,\ldots, 0) C(\alpha_0, \alpha_1, \ldots, \alpha_{n-1}).  \]
So we have, by Lemma~\ref{lem:isometry}
\[   \|B\|_{p,p}=\| H(\alpha_0,\alpha_1, \ldots, \alpha_{n-1}) \|_{p,p}
                  =\| C(\alpha_0,\alpha_1, \ldots, \alpha_{n-1}) \|_{p,p} . \]

We remark that $B=H(\alpha_0, \alpha_1,\ldots, \alpha_{n-1}) = {}^t H(\alpha_0, \alpha_1,\ldots, \alpha_{n-1})$.
Whenever $\alpha_i\in \R$ $(i=0,1,\ldots,n-1)$, we have
\[           \|B\|_{p,p}  = \|B^*\|_{q,q}= \|B\|_{q,q},  \] 
where $\frac{1}{p}+\frac{1}{q} =1$, and
\[   \|B\|_{2,2} = \min \{ \|B\|_{p,p} : p\in [1,\infty] \}  \]
by Corollary~\ref{cor:symmetry}.

%
%
\begin{corollary}
\begin{enumerate}
  \item[$(1)$] $B=H(\alpha_0, \alpha_1,\ldots, \alpha_{n-1}) \in LA(n)$ $\Leftrightarrow$ there exist $\beta, \omega \in \C$ such that $|\beta|=1$, $\omega^n=1$ and
\[  \alpha_i = \beta \omega^i |\alpha_i| \quad (i=0,1,\ldots, n-1).   \]
Then $\|B\|_{p,p} = \sum_{i=0}^{n-1} |\alpha_i|$ for any $p\in[1,\infty]$,
  \item[$(2)$] Let $A=C(\alpha_0, \alpha_1,\ldots, \alpha_{n-1})$ and $\alpha_i\in \R$ $(i=0,1,\ldots, n-1)$. 
Then we have
\begin{gather*}
  \|A\|_{2,2} = \max\{ |\sum_{i=0}^{n-1} \alpha_i \omega^i| : \omega \in \C \text{ with } \omega^n=1 \}  \\
  \text{ and } \|A\|_{2,2} \le \|A\|_{p,p} \le \begin{cases} (\sum_{i=0}^{n-1} |\alpha_i|)^{(2-p)/p}\|A\|_{2,2}^{2(p-1)/p} & p\in [1,2) \\
       (\sum_{i=0}^{n-1} |\alpha_i|)^{(p-2)/p}\|A\|_{2,2}^{2/p} & p\in [2,\infty] \end{cases} . 
\end{gather*}
\end{enumerate}
\end{corollary}

\noindent
{\bf Remark.}  In the case $\alpha_i \in \C$ $(0 \le i \le n-1)$, we can get the following estimation of $\|A\|_{p,p}$, 
where $A=C(\alpha_0,\alpha_1,\ldots, \alpha_{n-1})$:
We have already seen that the spectrum of normal matrix $A$ is equal to 
$\{ \sum_{j=0}^{n-1} \alpha_j\omega^j : \omega\in \C$ with  $\omega^n=1 \}$
in Proposition 3.6.
By Lemma~\ref{lem:inequality},
\begin{align*}
   & \|C(\alpha_0,\alpha_1,\ldots,\alpha_{n-1})\|_{p,p}  \le \|U\|_{2,p} \|U^*AU\|_{2,2} \|U^*\|_{p,2}  \\
    \le & n^{|1-\frac{2}{p}|} \max \{ |\sum_{j=0}^{n-1} \alpha_j\omega^j| : \omega\in \C \text{ with } \omega^n=1 \}.
\end{align*}


\section{Norms of block matrices}

Let $p\in [1,\infty]$, $\xi = \begin{pmatrix} x_1 \\ x_2 \\ \vdots \\ x_n \end{pmatrix}\in \ell^p(n)$  
and $\eta =\begin{pmatrix} y_1 \\ y_2 \\ \vdots \\ y_k \end{pmatrix}\in \ell^p(k)$.
We define $\xi\oplus \eta \in \ell^p(n+k)$ as follows:
\[   \xi \oplus \eta = \begin{pmatrix} \xi \\ \eta \end{pmatrix}  \]
and $\| \xi \oplus \eta \|_p = (|x_1|^p + \cdots + |x_n|^p + |y_1|^p+\cdots +|y_k|^p ) ^{1/p} = (\|\xi\|_p^p + \|\eta\|_p^p)^{1/p}$.
When $\xi= \oplus_{i=1}^k \xi_i = \begin{pmatrix} \xi_1 \\ \xi_2 \\ \vdots \\ \xi_k \end{pmatrix}$ and
$\xi_i = \begin{pmatrix} x^i_1 \\ x^i_2 \\ \vdots \\ x^i_n \end{pmatrix} \in \ell^p(n)$ $(i=1,2,\ldots, k)$,
we have
$\xi \in \ell^p(kn)$ and $\|\xi\|_p = (\sum_{i=1}^k \|\xi_i\|_p^p)^{1/p}=(\sum_{i=1}^k \sum_{j=1}^n |x^i_j|^p)^{1/p}$.
 
%
%
\begin{lemma}
Let $p\in [1,\infty]$ and $n, m\in \N$ with $n<m$.
\begin{enumerate}
  \item[$(1)$]  The mapping $\iota : \ell^p(n) \ni \xi \mapsto \iota(\xi)= \xi \oplus 0_{m-n} = \begin{pmatrix} \xi \\ 0 \end{pmatrix}\in \ell^p(m)$ is isometric.
  \item[$(2)$]  For $A\in \M_n(\C)$, we define $\iota(A) \in \M_m(\C)$ as follows:
\[    \iota(A) \zeta =\iota(A) (\xi\oplus \eta)= A\xi \oplus 0, \]
where $\zeta\in \ell^p(m), \xi\in\ell^p(n) \text{ and } \eta\in\ell^p(m-n)$. 
Then 
\[   \|\iota(A)\|_{p,p} = \|A\|_{p,p}.   \]
\end{enumerate}
\end{lemma}
\begin{proof}
(1) It is clear.

(2) Using (1), we have
\begin{align*}
  \|\iota(A)\|_{p,p} & = \sup \{ \|\iota(A)\zeta\|_p : \|\zeta\|_p=1\}  \\
    & = \sup \{ \|\iota(A)(\xi\oplus 0)\|_p : \|\xi\|_p=1\} \\
    & =\sup\{ \|A\xi\|_p : \|\xi\|_p=1\} = \|A\|_{p,p}.  
\end{align*}
\end{proof}

For $A\in \M_{n,m}(\C)$, we can see that $A$ is a linear map from $\ell^p(m)$ to $\ell^p(n)$.
If $n>m$ (respectively $n<m$), we define $\tilde{A}(\in \M_{\max\{m,n\}}(\C) )$ as follows:
\begin{gather*}
    \tilde{A}(\xi\oplus \eta) =A\xi  \in \ell^p(n)   \\
    (\text{respectively } \tilde{A}\xi = A\xi \oplus 0_{m-n} \in \ell^p(n)  ) 
\end{gather*}
for any $\xi\in\ell^p(m)$.
Then $\|A\|_{p,p} = \|\tilde{A}\|_{p,p}$, where
\begin{gather*}
    \|A\|_{p,p} = \sup \{ |\langle A\xi, \eta \rangle| : \xi \in \ell^p(m), \eta \in \ell^p(n), \|\xi\|_p=1, \|\eta\|_p=1 \},  \\
\intertext{and}
    \|\tilde{A}\|_{p,p} = \sup \{ |\langle A\xi, \zeta \rangle| : \xi, \zeta \in \ell^p(\max\{ m, n \}),  \|\xi\|_p=1, \|\zeta\|_p=1 \} 
\end{gather*} 
by the similar reason as above.

%
%
\begin{theorem}
Let $A\in \M_n(\C)$, $B\in \M_m(\C)$ and  $C=A\oplus B\in \M_{n+m}(\C)$.
\begin{enumerate}
  \item[$(1)$]  $\|C\|_{p,p} = \max\{ \|A\|_{p,p}, \|B\|_{p,p}\}$ for any $p\in [1,\infty]$.
  \item[$(2)$]  If  $C \in LA(n+m)$ then $A\in LA(n)$ or $B\in LA(m)$. 
Moreover we have $\|A\|_{p,p}=\|C\|_{p,p}$ for all $p\in[1,\infty]$ or $\|B\|_{p,p} =\|C\|_{p,p}$ for all $p\in[1,\infty]$.
\end{enumerate}
\end{theorem} 
\begin{proof}
(1)  It is clear that $\|C\|_{p,p} \ge \max\{\|A\|_{p,p}, \|B\|_{p.p} \}$.

By definition, there is $\zeta\in\ell^p(n+m)$ with $\|\zeta\|_p=1$ and $\|C\|_{p,p} =\|C\zeta\|_p$.
Then $\zeta$ can be decomposed by $\zeta=\xi\oplus \eta\in \ell^p(n)\oplus \ell^p(n)$ with $\|\xi\|_p^p + \|\eta\|_p^p = 1$.
Since
\begin{align*}
  \|C\|_{p,p} & = \|C\zeta\|_p = \|A\xi \oplus B\eta\|_p = (\|A\xi\|_p^p + \|B\eta\|_p^p)^{1/p}  \\
   & \le ( \|A\|_{p,p}^p\|\xi\|_p^p + \|B\|_{p,p}^p\| \eta \|_p^p)^{1/p} = ( \max\{ \|A\|_{p,p}^p, \|B\|_{p,p}^p\} )^{1/p} \\
   & = \max \{ \|A\|_{p,p}, \|B\|_{p,p} \},
\end{align*} 
we have done.

(2) When $\|A\|_{1,1} < \|C\|_{1,1}$ or $\|A\|_{\infty, \infty} < \|C\|_{\infty,\infty}$, we have, for any $p\in (1,\infty)$,
\[  \|A\|_{p,p} \le \|A\|_{1,1}^{1/p} \|A\|_{\infty,\infty}^{1-1/p} <  \|C\|_{1,1}^{1/p} \|C\|_{\infty,\infty}^{1-1/p}=\|C\|_{p,p}.  \]

If $\|A\|_{p_0,p_0} =\|C\|_{p_0,p_0}$ for some $p_0\in (1,\infty)$, then, by the above fact, it holds $\|A\|_{1,1}=\|C\|_{1,1}$ and 
$\|A\|_{\infty,\infty}=\|C\|_{\infty,\infty}$.
This means
\[  \|A\|_{p_0,p_0} =\|C\|_{p_0,p_0}= \|C\|_{1,1}^{1/p_0} \|C\|_{\infty,\infty}^{1-1/p_0}=\|A\|_{1,1}^{1/p_0} \|A\|_{\infty,\infty}^{1-1/p_0}. \]
So $A\in LA(n)$ by Corollary~\ref{cor:logaff}.

If $\|A\|_{p,p} \neq \|C\|_{p,p}$ for any $p\in (1,\infty)$, then $\|B\|_{p,p} = \|C\|_{p.p}$ for any $p\in[1,\infty]$ by (1).
So we have $B\in LA(m)$.
\end{proof}

%
%
\begin{proposition}\label{prop:norms}
Let $A_1, A_2,\dots, A_k \in \M_n(\C)$ and $B, C \in \M_{kn}(\C)$ as follows:
\[   B = \begin{pmatrix} A_1 & 0 & \cdots & 0 \\
                               A_2 & 0 & \cdots & 0 \\
                               \vdots & \vdots & \ddots & \vdots \\
                               A_k & 0 & \cdots & 0 \end{pmatrix}, \quad
     C = \begin{pmatrix} A_1 & A_2 & \cdots & A_k  \\
                               0 & 0 & \cdots & 0 \\
                               \vdots & \vdots & \ddots  & \vdots \\
                               0 & 0 & \cdots & 0 \end{pmatrix} .  \]
Then we have the following:
\begin{enumerate}
  \item[$(1)$] For  $p, q\in [1,\infty]$ with $\frac{1}{p}+\frac{1}{q}=1$,
\[ \|B\|_{p,p} \le (\sum_{i=1}^k \|A_i\|_{p,p}^p)^{1/p}  \text{ and } \|C\|_{p,p} \le ( \sum_{i=1}^k \|A_i\|_{p,p}^{q})^{1/q} .\]
 \item[$(2)$] If there exists $\xi \in \ell^p(n)$ such that $\|\xi\|_p=1$ and $\|A_i\xi\|_p = \|A_i\|_{p,p}$ $(1\le i \le k)$, then it holds
\[   \|B\|_{p,p} = (\sum_{i=1}^k \|A_i\|_{p,p}^p)^{1/p}  .  \]
 \item[$(3)$] If there exists $\eta\in \ell^{q}(n)$ such that $\|\eta\|_{q}=1$ and $\|A_i^*\eta\|_{q}= \|A_i^*\|_{q, q}$ $(1\le i \le k)$, then it holds
\[   \|C\|_{p,p} =(\sum_{i=1}^k \|A_i\|_{p,p}^{q})^{1/q}  .  \]
\end{enumerate}
\end{proposition}
\begin{proof}
(1)  Let $\zeta = \xi_1\oplus \xi_2\oplus \cdots \oplus \xi_k \in \ell^p(kn)$ with $\|\zeta\|_p=(\sum_{i=1}^k \|\xi_i\|_p^p)^{1/p}=1$.
Since
\begin{align*}
  \| B\zeta \|_p & = \| A_1\xi_1 \oplus A_2\xi_1 \oplus \cdots \oplus A_k\xi_1 \|_p  \\
   & = (\sum_{i=1}^k \|A_i\xi_1\|_p^p)^{1/p}  \le (\sum_{i=1}^k \|A_i\|_{p,p}^p)^{1/p},
\end{align*}
we have $\|B\|_{p,p} \le  (\sum_{i=1}^k \|A_i\|_{p,p}^p)^{1/p}$.

We remark that $\|C\|_{p,p}= \|C^*\|_{q, q}$ and 
$C^* = \begin{pmatrix} A_1^* & 0 & \cdots & 0 \\ A_2^* & 0 & \cdots & 0 \\ \vdots & \vdots & \ddots & \vdots \\ A_k^* & 0 & \cdots & 0 \end{pmatrix}$.
Then we have
\[  \|C\|_{p,p}  = \|C^*\|_{q, q} \le (\sum_{i=1}^k \|A_i^*\|_{q, q}^{q})^{1/q} =  (\sum_{i=1}^k \|A_i\|_{p,p}^{q})^{1/q}.  \]

(2)  Put $\zeta=\xi\oplus 0 \oplus \cdots \oplus 0 \in \ell^p(kn)$. 
Then we have $\|\zeta\|_p=1$ and
\[  \|B\zeta\|_p = (\sum_{i=1}^k \|A_i\xi\|_p^p)^{1/p}  = (\sum_{i=1}^k \|A_i\|_{p,p}^p)^{1/p},  \]
This means $\|B\|_{p,p} =  (\sum_{i=1}^k \|A_i\|_{p,p}^p)^{1/p}$.

(3)  By (2), $\|C^*\|_{q, q} = (\sum_{i=1}^k \|A_i^*\|_{q, q}^{q})^{1/q}$.
This implies 
\[     \|C\|_{p,p} = (\sum_{i=1}^k \|A_i\|_{p,p}^{q})^{1/q}  \]
by the calculation in (1). 
\end{proof}

%
%
\begin{corollary}
Let $A_{ij}\in \M_n(\C)$ $(i=1,2,\ldots,k, j=1,2,\ldots,l)$ and
\[   B = \begin{pmatrix} A_{11} & \cdots & A_{1l} \\ \vdots & & \vdots \\ A_{k1} & \cdots & A_{kl} \end{pmatrix} \in \M_{kn, ln}(\C) \subset \M_{\max\{k,l\}n}(\C). \]
Then, for $p,q\in [1,\infty]$ with $\frac{1}{p}+\frac{1}{q}=1$,
\[  \|B\|_{p,p} \le \min \{ (\sum_{i=1}^k(\sum_{j=1}^l \|A_{ij}\|_{p,p}^{q})^{p/q})^{1/p}, (\sum_{j=1}^l(\sum_{i=1}^k \|A_{ij}\|_{p,p}^p)^{q/p})^{1/q} \}.  \]
\end{corollary}
\begin{proof}
Put
\[  C_i =\begin{pmatrix} A_{i1} & \cdots & A_{il} \end{pmatrix} \in \M_{n, ln}(\C) \subset \M_{ln}(\C), \; i=1,2,\ldots,k.  \]
Then we can see
\[  B = \begin{pmatrix} C_1 \\ \vdots \\ C_k \end{pmatrix}  . \]
By Proposition~\ref{prop:norms} we have
\[  \|C_i\|_{p,p} \le (\sum_{j=1}^l \|A_{ij}\|_{p,p}^{q})^{1/q}   \]
and
\[  \|B\|_{p,p} \le (\sum_{i=1}^k \|C_i\|_{p,p}^p)^{1/p} \le (\sum_{i=1}^k(\sum_{j=1}^l \|A_{ij}\|_{p,p}^{q})^{p/q})^{1/p}.  \]

We also put
\[  D_j= \begin{pmatrix} A_{1j} \\ \vdots \\ A_{kj} \end{pmatrix} \in \M_{kn,n}(\C)\subset \M_{kn}(\C), \; j=1,2,\ldots,l, \]
and
\[  B = \begin{pmatrix} D_1 & \cdots &D_l \end{pmatrix}.  \]
Then we have
\[  \|D_j\|_{p,p} \le (\sum_{i=1}^k \|A_{ij}\|_{p,p}^p)^{1/p} ,  \]
and
\[  \|B\|_{p,p} \le (\sum_{j=1}^l \|D_j\|_{p,p}^{q})^{1/q} = (\sum_{j=1}^l(\sum_{i=1}^k \|A_{ij}\|_{p,p}^p)^{q/p})^{1/q} .  \]
So we can get the desired result.
\end{proof}

For a vector $\xi=\begin{pmatrix} x_1 \\ \vdots \\ x_n \end{pmatrix} \in \C^n$, we define maps $C,R: \C^n \rightarrow M_n(\C)$ as follows:
\[   C(\xi) =\begin{pmatrix} x_1 & 0 & \cdots & 0 \\ \vdots & \vdots & \cdots & \vdots \\ x_n & 0 & \cdots & 0 \end{pmatrix}, \quad
     R(\xi) = \begin{pmatrix} \bar{x_1} & \cdots & \bar{x_n} \\ 0 & \cdots & 0 \\ \vdots & \vdots & \vdots \\ 0 & \cdots & 0 \end{pmatrix}. \]
Then we have, by Proposition~\ref{prop:norms},
\[  \|C(\xi)\|_{p,p} = \|\xi\|_p  \text{ and } \|R(\xi)\|_{p,p} = \|\xi\|_{q} \qquad (\frac{1}{p}+\frac{1}{q}=1). \]

%
%
\begin{lemma}
For $\xi=\begin{pmatrix} x_1 \\ \vdots \\ x_n \end{pmatrix} \in \C^n$, the following are equivalent:
\begin{enumerate}
  \item[$(1)$] $C(\xi) \in LA(n)$.
  \item[$(2)$] $R(\xi) \in LA(n)$.
  \item[$(3)$] For any $i, j$,  $x_ix_j\neq 0$ implies $|x_i|=|x_j|$.
\end{enumerate}
\end{lemma}
\begin{proof}
$(1)\Leftrightarrow(2)$ It easily follows from the fact $\|C(\xi)\|_{p,p}=\|C(\xi)^*\|_{q, q}$, where $\frac{1}{p}+\frac{1}{q}=1$. 

$C(\xi)\in LA(n)$ $\Leftrightarrow$ $\|C(\xi)\|_{2,2} =\|C(\xi)\|_{1,1}^{1/2} \|C(\xi)\|_{\infty,\infty}^{1/2} $ (by Corollary~\ref{cor:logaff}) 
\[ \Leftrightarrow \sum_{i=1}^n |x_i|^2 = (\sum_{i=1}^n |x_i|) (\max_{1\le i \le n}|x_i| ) \Leftrightarrow (x_ix_j\neq 0 \text{ implies } |x_i|=|x_j| ). \]
So we have $(1)\Leftrightarrow (3)$.
\end{proof}

%
%
\begin{theorem}
Let $\alpha = \begin{pmatrix} a_1 \\ \vdots \\ a_n \end{pmatrix}$,  $\beta = \begin{pmatrix} b_1 \\ \vdots \\ b_n \end{pmatrix}\in \C^n$, and $A\in \M_n(\C)$.
The matrix $B\in \M_{n^2}(\C) = \M_n(\M_n(\C))$ is defined as follows:
\[   B=\begin{pmatrix} a_1\bar{b_1}A & \cdots & a_1\bar{b_n}A \\ \vdots & \ddots & \vdots \\ a_n\bar{b_1}A & \cdots & a_n\bar{b_n}A \end{pmatrix}.  \]
Then we have
\begin{enumerate}
  \item[$(1)$]  $\|B\|_{p,p} = \|\alpha\|_p \|\beta\|_{q} \|A\|_{p,p}$, where $\frac{1}{p}+\frac{1}{q}=1$.
  \item[$(2)$]  $B\in LA(n^2)$ is equivalent to all matrices $C(\alpha), R(\beta), A$ belongs to $LA(n)$.
\end{enumerate}
\end{theorem}
\begin{proof}
(1)  Put $C_i = \begin{pmatrix} a_i\bar{b_1}A & \cdots & a_i\bar{b_n}A \end{pmatrix} \in \M_{n,n^2}(\C) \subset \M_{n^2}(\C)$ $(1\le i \le n)$.
Since $\{ a_i\bar{b_1}A, \ldots, a_i\bar{b_n}A \}$ satisfies the condition (3) in Proposition~\ref{prop:norms},
\[   \|C_i\|_{p,p} =( \sum_{j=1}^n \| a_i \bar{b_j} A\|_{p,p}^{q} )^{1/q} =  |a_i| \|\beta\|_{q} \|A\|_{p,p} .  \]
Also $\{ C_1, \ldots C_n \}$ satisfies the condition (2) in Proposition~\ref{prop:norms},
\[   \|B\|_{p,p} = (\sum_{i=1}^n \|C_i\|_{p,p}^p)^{1/p} = (\sum_{i=1}^n (|a_i| \|\beta\|_{q} \|A\|_{p,p})^p)^{1/p} =\|\alpha\|_p \|\beta\|_{q} \|A\|_{p,p}.  \]

(2) We remark that, by Proposition~\ref{prop:riesz},
\begin{gather*}
  \|\alpha\|_2 \le (\|\alpha\|_1 \alpha\|_\infty)^{1/2}  \\
  \|\beta\|_2 \le (\|\beta\|_1 \|\beta\|_\infty)^{1/2}  \\
  \|A\|_{2,2} \le (\|A\|_{1,1} \|A\|_{\infty, \infty})^{1/2}  \\
  \|B\|_{2,2} \le (\|B\|_{1,1} \|B\|_{\infty, \infty})^{1/2}. 
\end{gather*}
By (1),  
\begin{align*}
   \|B\|_{2,2} & =\|\alpha\|_2 \|\beta\|_2 \|A\|_{2,2} \le  (\|B\|_{1,1} \|B\|_{\infty, \infty})^{1/2}  \\
   & = (\|\alpha\|_1 \|\alpha\|_\infty)^{1/2} (\|\beta\|_\infty \|\beta\|_1)^{1/2} (\|A\|_{1,1} \|A\|_{\infty, \infty})^{1/2}.
\end{align*}
So we have that $B\in LA(n^2)$ $\Leftrightarrow$ $\|B\|_{2,2} = (\|B\|_{1,1} \|B\|_{\infty, \infty})^{1/2}$ 
\begin{align*}
  \Leftrightarrow & \; \|\alpha\|_2  = (\|\alpha\|_1 \alpha\|_\infty)^{1/2},  
  \|\beta\|_2 = (\|\beta\|_1 \|\beta\|_\infty)^{1/2}  \\
     & \qquad  \qquad \text{ and }
  \|A\|_{2,2} = (\|A\|_{1,1} \|A\|_{\infty, \infty})^{1/2}  \\
   \Leftrightarrow & \; C(\alpha), R(\beta), A\in LA(n).
\end{align*}
\end{proof}

\noindent
{\bf Example}(4)  Let $A=\begin{pmatrix} 1 & 3 \\ 3 & 1 \end{pmatrix}$. 
By Example(1), $\|A\|_{p,p} = 4$ for any $p\in [1,\infty]$. 
In particular, $A\in LA(2)$.

For $\alpha= \begin{pmatrix} 1 \\ -1 \end{pmatrix}$ and $\beta=\begin{pmatrix} 1 \\ 2 \end{pmatrix}$, we consider
\[  B = \begin{pmatrix} (1\cdot 1) A &  (1\cdot 2) A \\ (-1\cdot 1)A & (-1\cdot 2)A \end{pmatrix}
       = \begin{pmatrix} 1 & 3 & 2 & 6 \\ 3 & 1 & 6 & 2 \\ -1 & -3 & -2 & -6 \\ -3 & -1 & -6 & -2 \end{pmatrix}.  \]
By Theorem 4.6 and Lemma 4.5, we have $\| B\|_{p,p}= 4\cdot 2^{1/p}\cdot (1+2^{1-1/p})^{1-1/p}$ and $B\notin LA(4)$ because
$R(\beta)\notin LA(2)$.

For $\alpha= \begin{pmatrix} 1 \\ i \\ 0 \end{pmatrix}$ and $\beta=\begin{pmatrix} 1 \\ -1 \\ i \end{pmatrix}$, we consider
\begin{align*}
   B & = \begin{pmatrix} (1\cdot 1) A &  (1\cdot (-1)) A & (1\cdot (-i))A \\ (i \cdot 1)A & (i \cdot (-1))A & (i\cdot (-i) A \\ 0\cdot A & 0 \cdot A & 0 \cdot A \end{pmatrix} \\
      & = \begin{pmatrix} 1 & 3 & -1 & -3 & i & -3i \\ 3 & 1 & -3 & -1 & -3i & -i \\ i & 3i & -i & -3i & 1 & 3 \\ 3i & i & -3i & -i & 3 & 1 \\
            0 & 0 & 0 & 0 & 0 & 0 \\ 0 & 0 & 0 & 0 & 0 & 0  \end{pmatrix}. 
\end{align*}
By Theorem 4.6 and Lemma 4.5, we have $\| B\|_{p,p}= 4\cdot 2^{1/p}\cdot 3^{1-1/p}$ and $B\in LA(6)$ because
$C(\alpha), R(\beta)\in LA(3)$ and $A\in LA(2)$.


\end{document}